\documentclass[11pt]{amsart}
\usepackage{latexsym}
\usepackage{amscd}
\usepackage{amssymb}
\usepackage[all, cmtip]{xy}
\usepackage{amsmath}%, pb-diagram}
\DeclareMathAlphabet{\eusm}{OT1}{eusm}{m}{n}
\newtheorem{theorem}{Theorem}[section]
\newtheorem{prop}[theorem]{Proposition}

\newtheorem{cor}[theorem]{Corollary}
\newtheorem{lem}[theorem]{Lemma}
\newtheorem{exam}[theorem]{Example}
\newtheorem{remark}[theorem]{Remark}

\def\Im{\mbox{Im\/}}
\def\End{\mbox{End\/}}

\def\H{\mbox{Hom\/}}
\def\M{\mbox{Mod\/}}

\begin{document}
\title{The Schr\"{o}der-Bernstein problem for Modules}
\author{Pedro A. Guil Asensio}
\address{Departamento de Mathematicas, Universidad de Murcia, Murcia, 30100, Spain}
\email{paguil@um.es}
\author{Berke Kalebo\~{g}az}
\address{Department of Mathematics, Hacettepe University, Ankara, 06800, Turkey}
\email{bkuru@hacettepe.edu.tr}
\author{Ashish K. Srivastava}
\address{Department of Mathematics and Statistics, St. Louis University, St.
Louis, MO-63103, USA} \email{asrivas3@slu.edu}

\maketitle

\begin{abstract}
In this paper we study the Schr\"{o}der-Bernstein problem for modules. We obtain a positive solution for
the Schr\"{o}der-Bernstein problem for modules invariant under
endomorphisms of their general envelopes under some mild conditions that are always satisfied, for example, in the case of injective,  pure-injective or cotorsion
envelopes. In the particular cases of injective envelopes and pure-injective envelopes, we are able to extend
it further and we show that the Schr\"{o}der-Bernstein problem has a
positive solution even for modules that are invariant only under
automorphisms of their injective envelopes or pure-injective envelopes.
\end{abstract}

\bigskip

\bigskip

\section{Introduction}

\noindent The Schr\"{o}der-Bernstein theorem is a classical result in
basic set theory. It states that if $A$ and $B$ are two sets such
that there is a one-to-one function from $A$ into $B$ and a one-to-one function from $B$ into $A$, then there exists a
bijective map between the two sets $A$ and $B$. This type of
problem where one asks if two mathematical objects $A$ and $B$
which are similar in some sense to a part of each other are also
similar themselves is usually called the Schr\"{o}der-Bernstein problem and it has been studied in various branches of Mathematics. The most
notable result along this direction is the one due to W. T. Gowers \cite{Gowers} where he constructed an example of two non-isomorphic Banach spaces such that each one is a complemented subspace of the other, thus showing that the Schr\"{o}der-Bernstein problem has a negative solution for Banach spaces. In the context of modules, this problem was studied by Bumby in \cite{Bumby} where he proved that
the Schr\"{o}der-Bernstein problem has a positive solution for modules which are invariant under endomorphisms of their injective envelope. 

The study of modules which are invariant under endomorphisms of their
injective envelope goes back to the pioneering work of
Johnson and Wong \cite{JW}. In order to prove that the Schr\"{o}der-Bernstein problem has a positive solution for modules which are invariant under endomorphisms of their injective envelope, Bumby first showed that if $M$ and $N$ are two modules such
that there is a monomorphism from $M$ to $N$ and a monomorphism
from $N$ to $M$, then their injective envelopes are isomorphic, that is, $E(M) \cong E(N)$. As a consequence,
he deduced that if $M$ and $N$ are two modules invariant under endomorphisms of their injective envelopes such that there is a monomorphism from $M$ to $N$ and a
monomorphism from $N$ to $M$, then $M \cong N$. 

Dickson and Fuller in \cite{DF} initiated the study of modules which are
invariant under all automorphisms of their injective envelope. Inspired by this, modules invariant under endomorphisms or, in particular,
automorphisms, of their general envelopes were recently introduced in
\cite{GKS}. The objective of this paper is to extend the result of Bumby for general envelopes and obtain Schr\"{o}der-Bernstein type results for modules invariant under endomorphisms or automorphisms of their envelopes.

Let $\mathcal X$ be a class of right $R$-modules
closed under isomorphisms and direct summands. An $\mathcal X$-{\em preenvelope} of a
right module $M$ is a homomorphism $u: M\rightarrow X$ with $X\in
\mathcal{X}$ such that any other homomorphism $g: M\rightarrow
X^\prime$ with $X^\prime \in \mathcal{X}$ factors through $u$. A
preenvelope $u: M\rightarrow X$ is called an $\mathcal{X}$-{\it
envelope} if it is minimal in the sense that any endomorphism
$h:X\rightarrow X$ such that $h \circ u=u$ must be an
automorphism. An $\mathcal{X}$-(pre)envelope $u: M\rightarrow X$
is called {\it monomorphic} if $u$ is a monomorphism. A class
$\mathcal{X}$ of right modules over a ring $R$, closed under
isomorphisms and direct summands, is called an {\em enveloping class} if any right
$R$-module $M$ has an $\mathcal X$-envelope. A module $M$ having a
monomorphic $\mathcal X$-envelope $u:M\rightarrow X(M)$ is said to
be {\em $\mathcal X$-automorphism invariant} (resp., {\em
$\mathcal X$-endomorphism invariant}) if for any automorphism
(resp., endomorphism) $\varphi:X(M)\rightarrow X(M)$, there exists
an endomorphism $f:M\rightarrow M$ such that $u\circ f=\varphi
\circ u$. It may be observed that when $\varphi$ is an automorphism, then $f$ also turns out to be an automorphism (see \cite{GKS}).

When $\mathcal X$ is the class of all injective modules, $\mathcal
X$-automorphism invariant modules are usually just called {\em
automorphism-invariant} modules and $\mathcal X$-endomorphism
invariant modules are called {\em quasi-injective} modules. When
$\mathcal X$ is the class of pure-injective modules, $\mathcal
X$-automorphism invariant modules are usually just called {\em
pure-automorphism-invariant} modules and $\mathcal X$-endomorphism
invariant modules are called {\em pure-quasi-injective} modules.

Let $\mathcal X$ be an enveloping class and $M, N$, two $\mathcal
X$-endomorphism invariant modules with monomorphic $\mathcal
X$-envelopes. Assume that $N$ is $\mathcal X$-strongly purely
closed and $M$ is an $\mathcal X$-strongly pure submodule of $N$ (see Section 2 for the definition of these concepts).
In this paper we show that if there exists an $\mathcal
X$-strongly pure monomorphism $u:N\rightarrow M$, then $M\cong N$.
In particular, this shows that the Schr\"{o}der-Bernstein property
holds for modules invariant under endomorphisms of their injective
or pure-injective envelopes or for flat modules invariant under
endomorphisms of their cotorsion envelopes. In the last section of
this paper, we extend this result further for the particular cases
of injective envelopes and pure-injective envelopes. We show that a
Schr\"{o}der-Bernstein type result holds for modules that are
invariant only under automorphisms of their injective envelopes or
pure-injective envelopes.

Throughout this paper, all rings will be associative with a unit element, unless stated otherwise. By an $R$-module, we will always mean a unitary right module over a ring $R$. And we will denote by $\M$-$R$, the category of right $R$-modules. We refer to \cite{Faith, lam2, W} for any undefined notion used along this paper.  

\bigskip

\section{The Schr\"{o}der-Bernstein problem for $\mathcal X$-endomorphism invariant modules}

\noindent Let $\mathcal X$ be an enveloping class of
right $R$-modules. We will assume along this paper that  every right $R$-module $M$ has
a monomorphic $\mathcal X$-envelope that we are going to denote by
$v_M:M\rightarrow X(M)$.

\noindent Following the notation in \cite[page 14]{GH1}, we are going to say
that a homomorphism $u:N\rightarrow M$ of right $R$-modules is an
{\em $\mathcal X$-strongly pure monomorphism} if any
 homomorphism $f:N\rightarrow X$, with
$X\in\mathcal X$, extends to a homomorphism $g:M\rightarrow X$
such that $g\circ u=f$. Let us note that $\mathcal X$-strongly
pure monomorphisms are clearly closed under composition. Moreover,
if $X\in \mathcal X$, then any $\mathcal X$-strongly pure
monomorphism $u:X\rightarrow M$ splits.

The following characterization of $\mathcal X$-strongly pure
monomorphisms is straightforward and we state it without any
proof.

\begin{lem}
Let $u:N\rightarrow M$ be a homomorphism. Then the following are equivalent;
\begin{enumerate}
\item $u$ is an $\mathcal X$-strongly pure monomorphism.
\item $v_N:N\rightarrow X(N)$ factors through $u$.
\item The composition $v_M\circ u:N\rightarrow X(M)$ is an $\mathcal X$-preenvelope.
\end{enumerate}
\end{lem}

\noindent Observe that condition (2) implies that any $\mathcal
X$-strongly pure monomorphism is a monomorphism, as we are
assuming that every module has a monomorphic $\mathcal
X$-envelope.

A submodule  $N$ of $M$ will be called an {\em $\mathcal
X$-strongly pure submodule} if the inclusion map $i:N\rightarrow M$ is an
$\mathcal X$-strongly pure monomorphism.

 Given a right $R$-module $M$, we will denote by add[$M$] the class of all direct summands of finite direct sums of copies of $M$. And we will say that a module $M$ is {\em $\mathcal X$-strongly purely closed} if any direct limit of splitting monomorphisms among objects in add[$M$] is an $\mathcal X$-strongly pure monomorphism.

\begin{exam}\label{vnr}
Let us give some examples of $\mathcal X$-strongly purely closed modules in which we will be interested in along this paper.

\begin{enumerate}
\item Let $\mathcal X$ be the class of all injective modules. Then any module is $\mathcal X$-strongly purely closed.

\item Let $\mathcal X$ be the class of all pure-injective modules. Then any module is $\mathcal X$-strongly purely closed.

\item Let $(\mathcal F, \mathcal C)$ be a cotorsion pair cogenerated by a set (see \cite{GT}) and assume that $\mathcal F$ is closed under taking direct limits. Then it is known that every module has a monomorphic $\mathcal C$-envelope (see e.g. \cite{Xu}). It is easy to check that any object in $\mathcal F \cap \mathcal C$ is $\mathcal C$-strongly purely closed.
\end{enumerate}
\end{exam}

\noindent In the proposition below, we describe the endomorphism
ring of $\mathcal X$-strongly purely closed modules. Recall that a module $M$
is called \textit{cotorsion} if $\text{Ext}^1(F,M)=0$ for every flat module $F$. It was shown in \cite{GH} that if $M$
is a flat cotorsion right $R$-module and $S=\End(M_{R}),$ then
$S/J(S)$ is a von Neumann regular right self-injective ring and idempotents lift modulo $J(S)$.

\begin{prop}
Let $\mathcal X$ be a class of modules closed under isomorphisms
and assume any module has a monomorphic $\mathcal X$-envelope.
Then for any  $\mathcal X$-strongly purely closed module $X$, $\End(X)$ is a
right cotorsion ring.

In particular, $\End(X)/J(\End(X))$ is von
Neumann regular right self-injective and idempotents lift modulo
$J(\End(X))$.
\end{prop}

\begin{proof}
Let us call $S=\End(X)$. Take any short exact sequence
$0\rightarrow S_S \stackrel{u}{\rightarrow} L \rightarrow F
\rightarrow 0$ with $F$, a flat right $S$-module. As $F$ is flat, the sequence is pure and thus, the
induced sequence $0\rightarrow S_S \otimes X_R \rightarrow
L\otimes X_R \rightarrow F \otimes X_R \rightarrow 0$ is also pure
in $\M$-$R$. On the other hand, we know that $F$ is a direct limit
of a family of finitely generated projective modules. Say that
$F=\varinjlim P_i$. Let us denote by
$\delta_i:P_i\rightarrow F$ the canonical homomorphisms from $P_i$
to the direct limit. Taking pullbacks, we get the following
commutative diagrams

\bigskip
\[
\xymatrix{
0 \ar[r]^{} & S \ar[d]^{\cong} \ar[r]^{u_i} & L_i\ar[d]^{\varphi_i} \ar[r]^{} & P_i \ar[d]^{\delta_i} \ar[r]^{} & 0\\
0 \ar[r]^{} & S \ar[r]^{u} & L \ar[r]^{}  & F \ar[r]^{} & 0 }
\]

\bigskip

\noindent in which the upper row splits, since $P_i$ is projective.
Moreover, $L=\varinjlim L_i$.
 Applying now the functor
$-\otimes_{S} X$, we get the following commutative diagram in $\M$-$R$.

  \bigskip
\[
\xymatrix{
0 \ar[r]^{} & S\otimes_S X \ar[d]^{1_S \otimes X} \ar[r]^{u_i\otimes X} & L_i\otimes X\ar[d]^{\varphi_i \otimes X} \ar[r]^{} & P_i\otimes X \ar[d]^{\delta_i\otimes X} \ar[r]^{} & 0\\
0 \ar[r]^{} & S\otimes X \ar[r]^{u\otimes X} & L\otimes X \ar[r]^{}  & F\otimes X \ar[r]^{} & 0 }
\]

\bigskip

\noindent We have $L\otimes X=\varinjlim L_i\otimes X$ and
$F\otimes X=\varinjlim P_i\otimes X$, since $-\otimes_S X$
commutes with direct limits. Note that $S\otimes_S X\cong X$ and
$P_i \otimes X$ is isomorphic to a direct summand of a finite
direct sum of copies of $X$. This shows that $u\otimes X$ is a
direct limit of splitting monomorphisms among modules in $\text{add}[X]$
and, as we are assuming that $X$ is an $\mathcal
X$-strongly purely closed module, this means that $u\otimes X$ is an $\mathcal
X$-strongly pure monomorphism. So there exists an $h:L\otimes_S X\rightarrow
S\otimes X$ such that $h\circ (u\otimes X)=1_{S\otimes X}$.
Applying now the functor $\H_R(X, -)$, we get the following
diagram in $\M$-$S$,

 \bigskip
\[
\xymatrix{
S \ar[d]^{\sigma_S} \ar[r]^{u}  & L\ar[d]^{\sigma_L}\\
\H(X, S\otimes X) \ar[r]^{} & \H(X, L\otimes X) }
\]

\bigskip

\noindent in which $\sigma_S$ is an isomorphism, $\sigma_L\circ u=\H(X,
u\otimes X)\circ \sigma_S$, $\H(X, 1_{S\otimes X})\circ
\sigma_S=\sigma_S$ and $\H(X, h) \circ \H(X, u\otimes X)=\H(X,
1_{S\otimes X})$. Therefore, ${\sigma_S}^{-1} \circ \H(X, h) \circ
\sigma_L \circ u=1_S$ and this shows that $u$ splits. Thus, the short
exact sequence $0\rightarrow S_S\rightarrow ^{u} L \rightarrow F
\rightarrow 0$ splits and hence $\End(X)$ is a right cotorsion
ring. Finally, by \cite{GH},
$\End(X)/J(\End(X))$ is von Neumann regular right self-injective
and idempotents lift modulo $J(\End(X))$.
\end{proof}

\bigskip

\noindent We are now ready to prove our first theorem.

 \begin{theorem}
Let $X\in \mathcal X$ be an $\mathcal X$-strongly purely closed module and $Y \in \mathcal X$,
an $\mathcal X$-strongly pure submodule of $X$. If there exists
an $\mathcal X$-strongly pure monomorphism $u:X\rightarrow Y$, then $X\cong Y$.
 \end{theorem}

 \begin{proof}
As $Y\in \mathcal X$, $Y$ must be a
 direct summand of $X$. Thus, we can find a submodule $H$ of $X$ such
that $X=H\oplus Y$. Now $$X=H\oplus Y\supseteq H\oplus
u(X)=H\oplus u(H) \oplus u(Y) \supseteq \ldots$$ and thus, calling
$P=\oplus_{t=0}^{\infty} u^t(H)$, we get that $X\supseteq P$. By
construction, $P\cap Y=\oplus_{t=1}^{\infty} u^t(H)=u(P)$.

Let $v_{P\cap Y}:P\cap Y\rightarrow X(P\cap Y)$ be the $\mathcal
X$-envelope of $P\cap Y$ and call $w:P\cap Y\rightarrow Y$, the
inclusion. Note that $Y$ is an $\mathcal X$-strongly purely closed
module, since it is a direct summand of $X$. And, as $w$ is a
directed union of inclusions of direct summands of $Y$, it is an
$\mathcal X$-strongly pure monomorphism. This means that there
exists a $q:Y\rightarrow X(P\cap Y)$ such that $q\circ w=v_{P\cap
Y}$, since $X(P\cap Y)\in \mathcal X$. Similarly, as $Y \in
\mathcal X$ and $X(P\cap Y)$ is an $\mathcal X$-envelope there
exists an $h:X(P\cap Y)\rightarrow Y$ such that $h\circ v_{P\cap
Y}=w$.

In particular, $q\circ h \circ v_{P\cap Y}=v_{P\cap Y}$. As
$v_{P\cap Y}$ is an envelope, we deduce that $q\circ h$ is an
isomorphism. Therefore, $h$ is a splitting monomorphism and
$Q=\Im(h)$ is a direct summand of $Y$. So there exists a submodule
$K$ such that $Y=Q\oplus K$.

Now, $X=H\oplus Y=H\oplus (Q\oplus K)=(H\oplus Q)\oplus K$. Thus,
$H\oplus Q\in \mathcal X$. Moreover, the inclusion $i:P
\rightarrow H\oplus Q$ may be viewed as $i=(1_{H} \oplus v_{P\cap
Y}): P=H\oplus (P\cap Y)\rightarrow H\oplus X(P\cap Y)\cong
H\oplus Q$. So $i$ is an $\mathcal X$-strongly pure monomorphism.
Now, as $Q\in \mathcal X$, we deduce that there exists a
$\psi:H\oplus Q\rightarrow Q$ such that $\psi \circ i=h\circ
v_{P\cap Y} \circ u$. Thus $h^{-1} \circ \psi \circ i=v_{P\cap Y}
\circ u$. Note that $u:P\rightarrow P\cap Y$ and $h:X(P\cap
Y)\rightarrow Q$ are isomorphisms. By the same way, as $H\oplus
Q\in \mathcal X$, there exists a $\varphi: Q\rightarrow H\oplus Q$
such that $\varphi \circ h \circ v_{P\cap Y} \circ u=i$. This
gives us $\varphi\circ \psi \circ i=i$ and $\psi \circ \varphi
\circ h \circ v_{P\cap Y}=v_{P\cap Y}$. On the other hand, as
$v_{P\cap Y}: P\cap Y\rightarrow X(P\cap Y)$ is an $\mathcal
X$-envelope and $H\in \mathcal X$, we get that $i=(1_{H} \oplus
v_{P\cap Y}): P=H\oplus (P\cap Y)\rightarrow H\oplus Q$ is an
$\mathcal X$-envelope. As both $i$ and $v_{P\cap Y}$ are
envelopes, we deduce that $\varphi \circ \psi$ and $\psi \circ
\varphi \circ h$ (and thus, $\psi \circ \varphi$) are
automorphisms. Therefore, both $\varphi$ and $\psi$ are
isomorphisms. Finally, $\psi \oplus 1_{K}: X=(H\oplus Q) \oplus K
\rightarrow Q\oplus K=Y$ is the desired isomorphism. Thus, $X\cong
Y$.
 \end{proof}

Applying the above theorem to the particular cases of injective envelopes, pure-injective envelopes and cotorsion envelopes, we obtain the following.

\begin{cor} \label{envelopes} Let $E$ be a module.
\begin{enumerate}

%\medskip

\item $($Bumby, \cite{Bumby}$)$ If $E$ is an injective module and $E'$, an injective submodule of $E$ such that there exists a monomorphism $u:E\rightarrow E'$, then $E\cong E'$.

%\medskip

\item If $E$ is a pure-injective module and $E'$, a pure-injective pure submodule of $E$ such that there exists a pure monomorphism $u:E\rightarrow E'$, then $E\cong E'$.

%\medskip

\item If $E$ is a flat cotorsion module and $E'$, a pure submodule of $E$ such that $E'$ is also flat cotorsion and there exists a pure monomorphism $u:E\rightarrow E'$, then $E\cong E'$.
 \end{enumerate}
 \end{cor}

\begin{proof}
The above theorem applies to the cases of injective, pure-injective and flat cotorsion modules in view of Example~\ref{vnr}.
\end{proof}

 The following lemma will be used in our next theorem.

\begin{lem}
A direct summand of an $\mathcal
X$-endomorphism invariant module is also $\mathcal X$-endomorphism
invariant.
\end{lem}

\begin{proof}
Let $M$ be an $\mathcal X$-endomorphism invariant
module and $N$, a direct summand of $M$. So there exists a module
$K$ such that $M=N\oplus K$. Thus, $X(M)=X(N)\oplus X(K)$. Let
$f:X(N)\rightarrow X(N)$ be an endomorphism of $X(N)$. So
$\iota_{X(N)}\circ f\circ \pi_{X(N)}$ is an endomorphism of $X(M)$,
where $\iota_{X(N)}:X(N)\rightarrow X(M)$ is the inclusion and
$\pi_{X(N)}:X(M)\rightarrow X(N)$ is the canonical projection. We clearly have $v_N\circ \pi_N=
\pi_{X(N)}\circ v_M$ and $v_M\circ \iota_N= \iota_{X(N)}\circ v_N$
with $\iota_N:N\rightarrow M$, the inclusion and
$\pi_N:M\rightarrow N$, the canonical projection. As $M$ is
$\mathcal X$-endomorphism invariant, there exists $h:M\rightarrow
M$ such that $v_M\circ h=\iota_{X(N)}\circ f\circ \pi_{X(N)}\circ
v_M$. We deduce that $g=\pi_N\circ h\circ \iota_N:N\rightarrow N$
is an endomorphism of $N$ such that $v_N\circ g=f\circ v_N$. So
$N$ is an $\mathcal X$-endomorphism invariant module.
\end{proof}

Our next theorem addresses the Schr\"{o}der-Bernstein problem for modules invariant under endomorphisms of their general envelopes.

\begin{theorem}\label{endo}
Let $\mathcal X$ be an enveloping class and $M, N$, two $\mathcal X$-endomorphism invariant modules
with monomorphic
$\mathcal X$-envelopes $v_M:M\rightarrow X(M)$ and $v_N:N\rightarrow X(N)$, respectively. Assume that $N$ is $\mathcal X$-strongly purely closed and $M$ is an $\mathcal X$-strongly pure submodule of $N$. If there exists an $\mathcal X$-strongly pure monomorphism
$u:N\rightarrow M$, then $M\cong N$.
\end{theorem}

\begin{proof}
Let $w'$ be an $\mathcal X$-strongly pure monomorphism from $M$ to
$N$. As $v_M:M\rightarrow X(M)$ is an $\mathcal X$-envelope and
$v_N\circ w':M\rightarrow X(N)$ is an $\mathcal X$-preenvelope,
there exists a split monomorphism $f_1:X(M)\rightarrow X(N)$ such
that $f_1\circ v_M=v_N\circ w'$. Similar argument shows that there
exists a split monomorphism $f_2:X(N)\rightarrow X(M)$ such that
$f_2\circ v_N=v_M\circ u$. Since the composition $f_2\circ
f_1:X(M)\rightarrow X(M)$ is also a split monomorphism, there
exists an endomorphism $g:X(M)\rightarrow X(M)$ such that
$g\circ(f_2\circ f_1)=1_{X(M)}$. Moreover, as $M$ is $\mathcal
X$-endomorphism invariant, there exists a homomorphism
$\delta:M\rightarrow M$ such that $v_M\circ \delta=g\circ v_M$.
This gives us, $v_M\circ \delta\circ u\circ w'= g\circ v_M\circ u
\circ w'=g\circ f_2\circ v_N\circ w'=g\circ f_2 \circ f_1\circ
v_M=v_M$. Since $v_M$ is a monomorphism, $\delta\circ u\circ w'$
is an automorphism. Therefore, $w'$ is a splitting monomorphism
and this yields that $M$ is a direct summand of $N$. Thus, we can
find a submodule $H$ of $N$ such that $N=H\oplus M$. Now,
$$N=H\oplus
M\supseteq H\oplus u(N)=H\oplus u(H)\oplus u(M)\supseteq
\ldots\supseteq \oplus_{i=0}^{n}u^{i}(H)\oplus u^{n}(M)\supseteq
\ldots$$

Call
$P=\oplus_{i=0}^{\infty}u^{i}(H)=H\oplus(\oplus_{i=1}^{\infty}u^{i}(H))=H\oplus(P\cap
M)\subseteq N$. By construction, $u(P)=P\cap M$. Let $v_{P\cap
M}:P\cap M\rightarrow X(P\cap M)$ be an $\mathcal X$-envelope of
$P\cap M$ and $w:P\cap M\rightarrow M$ be the inclusion. As $w$ is
a directed union of inclusions of direct summands of $M$, it is an
$\mathcal X$-strongly pure monomorphism. As $v_{P\cap M}$ is an
$\mathcal X$-envelope, there exists a homomorphism $h:X(P\cap
M)\rightarrow X(M)$ such that $h\circ v_{P\cap M}=v_{M}\circ w$.
And as $w$ is an $\mathcal X$-strongly pure monomorphism, there
exists a homomorphism $p:X(M)\rightarrow X(P\cap M)$ such that
$p\circ v_M\circ w=v_{P\cap M}$. In particular, $p\circ h \circ
v_{P\cap M}=v_{P\cap M}$ and since $v_{P\cap M}$ is an $\mathcal
X$-envelope, $p\circ h=1_{X(P\cap M)}$. On the other hand, $h\circ
p$ is an endomorphism of $X(M)$. As $M$ is $\mathcal
X$-endomorphism invariant, $(h\circ p)(M)\subseteq M$. This means
that, $h|_{p(M)}$ is a homomorphism from $p(M)$ to $M$.

Now we proceed to show that, $v_{p(M)}:p(M)\rightarrow X(P\cap M)$
is an $\mathcal X$-envelope and $p(M)$ is $\mathcal
X$-endomorphism invariant. Let $X'\in \mathcal X$ and
$f:p(M)\rightarrow X'$ be a homomorphism. As $v_M$ is an $\mathcal
X$-envelope and $X'\in \mathcal X$, there exists a homomorphism
$\alpha:X(M)\rightarrow X'$ such that $\alpha\circ v_M=f\circ
p|_M$.  Note that, $v_M\circ h|_{p(M)}=h\circ v_{P(M)}$ and
$p\circ v_M=v_{P(M)}\circ p|_M$, by the definitions of the
homomorphisms. Therefore, we have $\alpha\circ h: X(P\cap
M)\rightarrow X'$ with $(\alpha\circ h)\circ v_{p(M)}=f$. So we
deduce that $v_{p(M)}:p(M)\rightarrow X(P\cap M)$ is an $\mathcal
X$-preenvelope. Moreover, it can be shown that
$v_{p(M)}:p(M)\rightarrow X(P\cap M)$ is indeed an $\mathcal
X$-envelope. Now, let $\varphi:X(P\cap M)\rightarrow X(P\cap M)$
be an endomorphism. As $h\circ \varphi\circ p$ is an endomorphism
of $X(M)$ and $M$ is $\mathcal X$-endomorphism invariant, $(h\circ
\varphi\circ p)(M)\subseteq M$. So we have $\varphi(p(M))\subseteq
p(M)$. Thus, $p(M)$ is $\mathcal X$-endomorphism invariant.

Furthermore, we have $v_{p(M)}=p\circ h\circ v_{p(M)}=p\circ
v_M\circ h|_{p(M)}= v_{p(M)}\circ p|_M \circ h|_{p(M)}$ and, as
$v_{p(M)}$ is a monomorphism, we get that $p|_M \circ
h|_{p(M)}=1_{p(M)}$. Therefore, $h|_{p(M)}:p(M)\rightarrow M$ is a
splitting monomorphism and $Q=Im(h|_{p(M)})=h\circ p(M)$ is a
direct summand of $M$. So there exists a module $K$ such that
$M=Q\oplus K$. Again, $N=H\oplus M=H\oplus (Q\oplus K)=(H\oplus
Q)\oplus K$ and thus $H\oplus Q$ is an $\mathcal X$-endomorphism
invariant module.

Moreover, the inclusion $i:P \rightarrow H\oplus Q$ may be viewed
as $i:=(1_{H} \oplus (p|_M\circ w)): P=H\oplus (P\cap
M)\rightarrow H\oplus p(M)\cong H\oplus Q$. So $i$ is an $\mathcal
X$-strongly pure monomorphism. As $H\oplus Q$ is $\mathcal
X$-endomorphism invariant, there exists a $\psi:H\oplus
Q\rightarrow Q$ such that $\psi \circ i=h|_{p(M)}\circ p|_M\circ
w\circ u|_P$, where $h|_{p(M)}:p(M)\rightarrow Q$ and $u|_P:P\rightarrow P\cap M$ are
isomorphisms. On the other hand, as $i$ is an $\mathcal
X$-strongly pure monomorphism and $H\oplus Q$ is an $\mathcal
X$-endomorphism invariant module, we get that $i: P=H\oplus (P\cap
M)\rightarrow H\oplus Q$ is an $\mathcal X$-envelope. Similarly,
there exists a homomorphism $\varphi: Q\rightarrow H\oplus Q$ such
that $\varphi \circ h|_{p(M)} \circ p|_M \circ w \circ u|_P=i$.
This means that $\varphi\circ \psi \circ i=i$ and
$h|_{p(M)}^{-1}\circ \psi \circ \varphi \circ h|_{p(M)} \circ
v_{P\cap M}=v_{P\cap M}$. And, as both $i$ and $v_{P\cap M}$ are
envelopes, we deduce that $\varphi \circ \psi$ and
$h|_{p(M)}^{-1}\circ \psi \circ \varphi \circ h|_{p(M)}$ are
automorphisms. Thus, it follows that $\psi \circ \varphi$ is also
an automorphism. Therefore, both $\varphi$ and $\psi$ are
isomorphisms. Finally, $\psi \oplus 1_{K}: N=(H\oplus Q) \oplus K
\rightarrow Q\oplus K=M$ is the desired isomorphism. This
completes the proof.
\end{proof}

Applying the above theorem to the particular cases of injective envelopes, pure-injective envelopes and cotorsion envelopes, we get the following.

\begin{cor} Let $M$ and $N$ be two modules.
%\medskip
\begin{enumerate}
\item $($Bumby, \cite{Bumby}$)$ If $M$ and $N$ are quasi-injective modules such that there is a monomorphism from $M$ to $N$ and a monomorphism from $N$ to $M$, then $M\cong N$.
%\medskip
\item If $M$ and $N$ are pure-quasi-injective modules such that there is a pure monomorphism from $M$ to $N$ and a pure monomorphism from $N$ to $M$, then $M\cong N$.
%\medskip
\item If $M$ and $N$ are flat modules invariant under endomorphisms of their cotorsion envelopes such that there is a pure monomorphism from $M$ to $N$ and a pure monomorphism from $N$ to $M$, then $M\cong N$.
\end{enumerate}
\end{cor}

\section{Schr\"{o}der-Bernstein problem for Automorphism-invariant modules}

\noindent Although we do not know if the results from previous section can be extended to modules invariant under automorphisms of their envelopes in the general case, we will study this question  for the particular case of injective and pure-injective envelopes in this section. Recently, it has been shown in \cite{AFT} that if $M$ and $N$ are
automorphism-invariant modules of finite Goldie dimension such
that there is a monomorphism from $M$ to $N$ and a monomorphism
from $N$ to $M$, then $M \cong N$. We will extend this result and
show that the Schr\"{o}der-Bernstein problem has a positive solution for any automorphism-invariant module.

We will denote the injective envelope of a module $M$ by $E(M)$ and $A\subseteq_{e} B$ will mean that $A$ is an essential submodule of $B$.
We can now prove the main result of this section.

\begin{theorem}\label{auto}
Let $M, N$ be automorphism-invariant modules and let $f:M\rightarrow N$ and $g:N\rightarrow M$ be monomorphisms. Then $M\cong N$.
\end{theorem}

\begin{proof}
By Corollary \ref{envelopes}, we know that $E(M)\cong E(N)$. On the other hand, we have a diagram

\bigskip
\[
\xymatrix{
M\ar[d]^{1_M}  \ar[r]^{f} &f(M) \ar[r]^{u} & N \ar[r]^{g} & M\\
M }
\]

\bigskip
\noindent in which $u:f(M) \rightarrow N$ is the inclusion. As $M$
is automorphism-invariant and $g\circ u\circ f$ is monic, there
exists a $\varphi:M\rightarrow M$ such that $\varphi \circ g \circ
u\circ f=1_M$ (see \cite{ESS} and \cite{preprint}). And, as $f:M\rightarrow f(M)$ is an isomorphism,
this means that $u:f(M) \rightarrow N$ splits and thus, $f(M)$ is
a direct summand of $N$. Similarly, $g(N)$ is a direct summand of
$M$.

As $f:M\rightarrow f(M)$ and $g:N\rightarrow g(N)$ are
isomorphisms, we know that $E(f(M))\cong E(g(N))$. We proceed to
show that $f(M) \cong g(N)$. Let $h:E(g(N))\rightarrow E(f(M))$ be
an isomorphism. Call $M'=h^{-1}(f(M))\cap g(N)$ and $N'=h(g(N))
\cap f(M)$. By construction $h|_{M'}:M'\rightarrow N'$ is an
isomorphism. Moreover, as $g(N)\subseteq_{e} E(g(N))$, we have that
$h(g(N))\subseteq_{e} E(f(M))$. Similarly, $h^{-1}(f(M)) \subseteq_{e}
E(g(N))$. Therefore, $M'\subseteq_{e} E(g(N))$ and $N'\subseteq_{e} E(f(M))$.
In particular, $M'\subseteq_{e} g(N)$ and $N'\subseteq_{e} f(M)$. We have then

\bigskip
\[
\xymatrix{
M'\ar[d]^{u_{M'}}  \ar[r]^{h|_{M'}} & N' \ar[r]^{u_{N'}} & f(M)\\
g(N) }
\]

\bigskip

\noindent where $u_{M'}$ and $u_{N'}$ are inclusions. Moreover,
$g(N)$ is a submodule of $M$ and $f(M)$ is isomorphic to $M$.
Therefore, $f(M)$ is automorphism-invariant and as, $u_{N'} \circ
h|_{M'}$ is monic, there exists a $\psi:g(N) \rightarrow f(M)$
such that $\psi \circ u_{M'}=u_{N'} \circ h|_{M'}$. Similarly,
there exists a $\varphi:f(M) \rightarrow g(N)$ such that $\varphi
\circ u_{N'}=u_{M'} \circ h^{-1}|_{N'}$. Composing, we get the
diagram

\bigskip
\[
\xymatrix{
M'\ar[d]^{u_{M'}}  \ar[r]^{h|_{M'}} & N' \ar[d]^{u_{N'}} \ar[r]^{h^{-1}|_{N'}} & M'\ar[d]^{u_{M'}}\\
g(N) \ar[r]^{\psi} & f(M) \ar[r]^{\varphi} & g(N) }
\]

\bigskip

\noindent So $\varphi \circ \psi \circ u_{M'}= \varphi \circ
u_{N'} \circ h|_{M'}=u_{M'} \circ h^{-1}|_{N'} \circ
h|_{M'}=u_{M'}$. And this means that $(1_{g(N)} - \varphi \circ
\psi) \circ u_{M'} =0$. As $u_{M'}$ is monic, we deduce that
$(1_{g(N)} - \varphi \circ \psi)$ has essential kernel and thus,
$(1_{g(N)} - \varphi \circ \psi) \in J(\End(g(N))$ since $g(N)$ is
automorphism-invariant. Therefore, $\varphi \circ \psi$ is an
isomorphism. Similarly, $\psi \circ \varphi$ is an isomorphism and
thus, $\varphi: f(M) \rightarrow g(N)$ is an isomorphism. As
$M\cong f(M)$ and $N\cong g(N)$, we deduce that $M\cong N$.
\end{proof}

\noindent Let us finish this paper by extending the above result to modules which are invariant under automorphisms of their pure-injective envelope. For that, recall that there exists a full embedding $H:{\rm Mod-}R \rightarrow \mathcal{D}$ of ${\rm Mod-}R$ into a locally finitely presented Grothendieck category $\mathcal D$ (normally
called the functor category of ${\rm Mod}$-$R$) satisfying the following key properties (see e.g. \cite{CB,Si}):

\begin{itemize}
\item $H$ has a right adjoint functor $G:\mathcal{D} \rightarrow {\rm Mod}$-$R$.

\item An exact sequence
$$\Sigma\equiv 0 \rightarrow X\rightarrow Y\rightarrow Z\rightarrow 0$$
in ${\rm Mod}$-$R$ is pure if and only if the induced sequence $H(\Sigma)$
is exact (and pure) in $\mathcal D$.

\item $H$ identifies ${\rm Mod}$-$R$ with the full subcategory of $\mathcal D$ consisting of the all
 FP-injective objects in $\mathcal D$ where an object $D\in \mathcal D$ is FP-injective if ${\rm Ext}^1(D',D)=0$ for every finitely presented object $D'\in\mathcal D$.

\item A module $M\in{\rm Mod}$-$R$ is pure-injective if and only if $H(M)$ is an injective object of $\mathcal D$. And $u:M\rightarrow PE(M)$ is the pure-injective envelope of $M$ if and only if $H(u):H(M)\rightarrow H(PE(M))$ is the injective envelope of $H(M)$ in $\mathcal D$.

\end{itemize}

\noindent On the other hand, the locally finitely presented Grothendieck category $\mathcal D$ is equivalent to the category of unitary right $S$-modules for a ring with enough idempotents $S$
(see e.g. \cite[52.5(2)]{W}). Recall that a non-unital ring $R$ is said to have enough idempotents if there exists a set of orthogonal idempotents $\{e_i\}_{i\in I}$  in the ring such that $R=\oplus_I e_iR=\oplus_{i\in I}Re_i$.
And a right $R$-module $M$ is called unitary if $MR=M$. We refer to \cite[Section 49]{W} for the categorical properties of these unitary modules.

It is easy to check that all the proofs in this paper work for unitary right modules over a ring with enough idempotents.  Therefore, identifying $\mathcal D$ with ${\rm Mod-}S$, we may apply Theorem~\ref{auto} to ${\rm Mod-}S$ to get:

\begin{cor} \label{p-auto}
Let $M, N$ be two modules invariant under automorphisms of their pure-injective envelopes let $f:M\rightarrow N$ and $g:N\rightarrow M$ be pure monomorphisms. Then $M\cong N$.
\end{cor}

\begin{proof}
In this case, $H(M), H(N)$ are automorphism-invariant objects in $\mathcal D$ and $H(f):H(M)\rightarrow H(N)$ and $H(g):H(N)\rightarrow H(M)$ are monomorphisms. So $H(M)\cong H(N)$ by Theorem~\ref{auto}. And this means that $M\cong G\circ H(M)$ is isomorphic to $N\cong G\circ H(N)$.
\end{proof}

\begin{remark}
The above results suggest that it might be possible to extend Theorem~\ref{endo} in last section to  $\mathcal X$-automorphism invariant modules for which the endomorphism ring of their $\mathcal X$-envelope is right cotorsion; for instance, to flat modules which are invariant under automorphisms of their cotorsion envelopes. However, our techniques do not seem to work in this more general setting. 
\end{remark}

In regard to this possible extension, our next example shows that we cannot expect to deduce this kind of result from Theorem \ref{auto}, Corollary \ref{p-auto} or results in Section 2. Our example shows that there exist flat modules which are invariant under automorphisms of their cotorsion envelopes but they are not invariant under endomorphisms of their cotorsion envelopes, nor under automorphisms of their injective or pure-injective envelopes and therefore, our results cannot be applied to these modules.

\begin{exam} \rm
Let $K$ be a field of characteristic zero and $S$, the $K$-algebra constructed in \cite[Section 2]{Z}. Then $S$ is a right artinian ring which is not right pure-injective. As $S_S$ is artinian, any right $S$-module is cotorsion and thus, is invariant under automorphisms of its cotorsion envelope. Assume that any direct sum of copies of $S_S$ is invariant under automorphisms of its pure-injective envelope. As ${\rm char}(K)=0$, this means that any direct sum of copies of $S_S$ is also invariant under endomorphisms of its pure-injective envelope $($see \cite{GKS}$)$ and thus, $H(S_S)$ is $\Sigma$-quasi-injective in the functor category $\mathcal D$. But then, $E(H(S_S))$ is $\Sigma$-injective $($see \cite{Faith}$)$ and this means that the pure-injective envelope of $S_S$ is $\Sigma$-pure-injective. Therefore, $S_S$ is also $\Sigma$-pure injective as it is a pure submodule of its pure-injective envelope, a contradiction. Thus we conclude that there exists an index set $I$ such that $S_S^{(I)}$ is not invariant under automorphisms of its pure-injective envelope. Call $M_S=S_S^{(I)}$.

Let now $R$ be the ring of all eventually constant sequences over $\mathbb F_2$, the field of two elements. It is known that $R$ is a von Neumann regular ring, and $R_R$ is an automorphism-invariant module which is not quasi-injective $($see \cite{GKS}$)$. Therefore, it cannot be invariant under endomorphisms of its cotorsion envelope, nor of its pure-injective envelope, either.

Let us consider the ring $R\times S$ and the right $R\times S$-module $R\times M$. Then:

\begin{enumerate}
\item $R\times M$ is flat and it is invariant under automorphisms of its cotorsion envelope, since so are $R_R$ and $M_R$.

\item $R\times M$ is not invariant under endomorphisms of its cotorsion envelope, since otherwise so would be $R_R$.

\item $R\times M$ is not invariant under automorphisms of its pure-injective  envelope, since otherwise so would be $M_R$.

\item $R\times M$ is not invariant under automorphisms of its injective envelope, since otherwise it would be quasi-injective, as $S$ is an algebra over a field of characteristic zero. And this would mean that $S_S$ would be injective, since it is a direct summand of $M_S$.
\end{enumerate}
\end{exam}

\bigskip 

\noindent {\bf Acknowledgment.} The work of the third author is partially supported by a grant from Simons Foundation (grant number 426367). Part of this work was done when the first and the third authors were visiting Harish-Chandra Research Institute, Allahabad, India. They would like to thank the institute and Dr. Punita Batra for warm hospitality.

\end{document}